\documentclass[12pt]{article}
\usepackage{amsmath,amssymb,amsthm}
\usepackage{enumerate}
\usepackage{graphicx}
\usepackage[T1]{fontenc}
\usepackage{caption}
\usepackage{hyperref}
\usepackage{cite}

\def\draw #1 by #2 (#3){
	\vbox to #2{
		\hrule width #1 height 0pt depth 0pt
		\vfill
		\special{picture #3} 
	}
}

\def\scaleddraw #1 by #2 (#3 scaled #4){{
		\dimen0=#1 \dimen1=#2
		\divide\dimen0 by 1000 \multiply\dimen0 by #4
		\divide\dimen1 by 1000 \multiply\dimen1 by #4
		\draw \dimen0 by \dimen1 (#3 scaled #4)}
}

\newtheorem{theorem}{Theorem}[section]
\newtheorem{example}[theorem]{Example}
\newtheorem{problem}[theorem]{Problem}

\newtheorem{corollary}[theorem]{Corollary}
\newtheorem{remark}[theorem]{Remark}
\usepackage{xcolor}
\usepackage{float}
\usepackage{graphicx,subcaption,lipsum}
\usepackage{hyperref}
\newtheorem{nt}{Note}

\newcommand{\singlespacing}{\let\CS=\@currsize\renewcommand{\baselinestretch}{1}\tiny\CS}
\newcommand{\oneandahalfspacing}{\let\CS=\@currsize\renewcommand{\baselinestretch}{1.25}\tiny\CS}
\newcommand{\doublespacing}{\let\CS=\@currsize\renewcommand{\baselinestretch}{1.35}\tiny\CS}

\newtheorem{rule-def}[theorem]{Rule}

\begin{document}
	\baselineskip 16pt
	\newcommand{\la}{\lambda}
	\newcommand{\si}{\sigma}
	\newcommand{\ol}{1-\lambda}
	\newcommand{\be}{\begin{equation}}
		\newcommand{\ee}{\end{equation}}
	\newcommand{\bea}{\begin{eqnarray}}
		\newcommand{\eea}{\end{eqnarray}}
	\baselineskip=0.30in
		\baselineskip=0.30in
        \begin{center}
		{\Large  \textbf{ On Coalition Graphs and Coalition Count of Graphs} }\\
		\vspace*{0.3cm} 
    \vspace*{0.3cm} 
	\end{center}
	\begin{center}
		Swathi Shetty$^{1}$, Sayinath Udupa N. V.$^{*,2}$,  B. R. Rakshith$^{3}$\\
		Department of Mathematics, Manipal Institute of Technology\\ Manipal Academy of Higher Education\\ Manipal, India.\\
    swathishetty6498@gmail.com; swathi.dscmpl2022@learner.manipal.edu$^{1}$\\
		sayinath.udupa@manipal.edu	$^{*,2}$\\
       ranmsc08@yahoo.co.in; rakshith.br@manipal.edu$^{3}$.\\
        26 Nov 2025
	\end{center}
   \baselineskip=0.20in
     \footnotetext{*Corresponding author}  
\begin{abstract} Let $G$ be graph with vertex set $V(G)$ and order $n$. A set $S \subseteq V(G)$ is a dominating set of a graph $G$ if every vertex in $V(G) \backslash S$ is adjacent to at least one vertex in $S$. A coalition in a graph $G$ consists of two disjoint sets of vertices $V_1$ and $V_2$,  neither of which is a dominating set but whose union $V_1 \cup V_2$ is a dominating set.  A coalition partition, abbreviated $c$-partition, in a graph $G$ is a vertex partition $\pi=\left\{V_1 , V_2,\dots, V_k\right\}$ such that every set $V_i$ of $\pi$ is either a singleton dominating set, or is not a dominating set but forms a coalition with another set $V_j$ in $\pi$. The sets $V_i$ and $V_j$ are coalition partners in $G$. The coalition number $C(G)$ equals the maximum order $k$ of a $c$-partition of $G$. For any graph $G$ with a $c$-partition $\pi=\left\{V_1,V_2,\dots,V_k\right\}$, the coalition graph $CG(G,\pi)$ of $G$ is a graph with vertex set $V_1,V_2,\dots, V_k$, corresponding one-to-one with the set $\pi$, and two vertices $V_i$ and $V_j$ are adjacent in $CG(G,\pi)$ if and only if the sets $V_i$ and $V_j$ are coalition partners in $\pi$. In~\cite{haynes2023coalition}, authors proved that for every graph $G$ there exist a graph $H$ and $c$-partition $\pi$ such that $CG(H,\pi)\cong G$, and raised the question: Does there exist a graph $H^*$ of smaller order $n^*$ and size $m^*$ with a $c$-partition $\pi^*$ such that $CG(H^*,\pi^*)\cong G$?. In this paper, we constructed a graph $H^*$ of small order and size and a $c$- partition $\pi^*$ such that $CG(H^*,\pi^*)\cong G$.  Recently, Haynes et al.~\cite{haynes2020introduction} defined the coalition count $c(G)$ of a graph $G$ as the maximum number of different coalition in any $c$-partition of $G$.  We characterize all graphs $G$ with $c(G)=1$. Further, imposing some suitable conditions on coalition number, we study the properties of coalition count of graph.
\end{abstract}
\noindent 
      \textbf{Mathematics Subject Classifications:} 05A18, 05C69.\\[1mm]
       \textbf{Keywords:} Dominating set, Coalition partition, Coalition number.
\section{Introduction }
Let $G = (V(G), E(G))$ be a graph with vertex set $V(G)$, edge set $E(G)$, order $n(G) = |V(G)|$, and size $m(G) = |E(G)|$. A set $S \subseteq V(G)$ is a \textit{dominating set} of a graph $G$ if every vertex in $V(G) \backslash S$ is adjacent to at least one vertex in $S$. A detailed account of domination in graphs is presented in the recent books on domination theory~\cite{haynes2013fundamentals,haynes2020topics}. The open neighborhood $N(v)$ of a vertex $v$ in $G$ is the set of vertices adjacent to $v$, while the closed neighborhood of $v$ is the set $N[v] =
\left\{v\right\} \cup N(v)$. A vertex of degree one is a \textit{pendant vertex}.\\ 
    The concept of coalitions in graphs was introduced by Haynes et.al.~\cite{haynes2020introduction} in 2020 as follows. A \textit{coalition} in a graph $G$ consists of two disjoint sets of vertices $V_1$ and $V_2$,  neither of which is a dominating set but whose union $V_1 \cup V_2$ is a dominating set.
     A \textit{coalition partition}, henceforth called a $c$-partition, in a graph $G$ is a vertex partition $\pi=\left\{V_1 , V_2,\dots, V_k\right\}$ such that every set $V_i$ of $\pi$ is either a singleton dominating set, or is not a dominating set but forms a coalition with another set $V_j$ in $\pi$. The \textit{coalition number} $C(G)$ equals the maximum order $k$ of a $c$-partition of $G$, and a $c$-partition of $G$ having order $C(G)$ is called a \textit{$C(G)$-partition}. In~\cite{haynes2021upper}, Haynes et al.\ established upper bounds on the coalition number in terms of the minimum and maximum degree. In~\cite{haynes2023coalition}, they proved that every graph is the coalition graph of some graph, while in~\cite{haynes2023coalition1}, they investigated coalition graphs of trees, paths, and cycles. Furthermore, in~\cite{haynes2020introduction}, the authors posed the open problem of characterizing all graphs $G$ of order $n$ with coalition number $C(G)=n$. This problem was partially addressed by Bakhshesh et al.~\cite{bakhshesh2023coalition}, where the authors characterized all graphs of order $n$ with $\delta(G) \leq 1$ that satisfy $C(G)=n$, and identified all trees whose coalition number equals $n-1$. In addition, several variants of coalition have been introduced and studied in~\cite{henning2025double,samadzadeh2025paired,alikhani2024total,alikhani2025independent}.\par
     In the paper~\cite{haynes2020introduction} while introducing coalitions, the authors suggest several related areas for future study, one of which is coalition count of $G$ defined as follows. The coalition count $c(G)$ of a graph $G$ is equal to  the maximum number of different coalition in any $c$-partition of $G$.
    \begin{example}
      Consider the cycle $C_4=(v_1,v_2,v_3,v_4)$.\\ The partition $\pi_1=\left\{\left\{v_1\right\},\left\{v_2\right\},\left\{v_3\right\},\left\{v_4\right\}\right\}$ is a $c$-partition of $C_4$. No individual set in $\pi_1$ is a dominating set. However the following different coalition exist:
      \begin{itemize}
          \item  The set $\left\{v_1\right\}$ forms coalition with $\left\{v_2\right\}$, $\left\{v_3\right\}$ and $\left\{v_4\right\}$.
          \item  The set $\left\{v_2\right\}$ forms coalition with $\left\{v_3\right\}$ and $\left\{v_4\right\}$.
          \item The set $\left\{v_3\right\}$ form coalition with $\left\{v_4\right\}$.
      \end{itemize}
  Thus, every set in \( \pi_1 \) forms a coalition with at least one other set. Hence, \( C(C_4) = 4 \), and \( \pi_1 \) is a \( C(C_4) \)-coalition partition. But, the total number of different possible coalitions is 6, thus $c(C_4)=6$.  
    \end{example}
    \section{Coalition graphs}
    In~\cite{haynes2023coalition} Haynes et al. defined coalition graph of $G$ as follows.
   Let $G$ be a graph with a $c$-partition $\pi=\left\{V_1,V_2,\dots,V_k\right\}$. The coalition graph $CG(G,\pi)$ of $G$ is a graph with vertex set $V_1,V_2,\dots, V_k$,  and two vertices $V_i$ and $V_j$ are adjacent in $CG(G,\pi)$ if and only if the sets $V_i$ and $V_j$ are coalition partners in $\pi$. 
	In~\cite{haynes2023coalition} Haynes et.al proved that, for every graph $G$, there is a graph $H$ and some $c$-partition $\pi$ of $H$, such that $CG(H,\pi)\cong G$, that is, for every graph $G=(V,E)$ having $n$ non-isolated vertices and $t$ isolated vertices with $|E|=m$, they constructed a graph $H$ of order $n+m+\binom{n}{2}+t$ and size $m=\binom{n}{2}+2m(n-1)+\overline{m}(n-2)+t(n(H)-1)$, where $n(H)$ is order of $H$, and the $c$-partition $\pi$ such that $CG(H,\pi)\cong G$.\\
	 Further, they raised the following question: Does there exist a graph $H^*$ of smaller order $n^*$ and size $m^*$ with a $c$-partition $\pi^*$ such that $CG(H^*,\pi^*)\cong G$?.
	\begin{theorem}
		For every graph $G$, there exist a graph $H^*$ and $c$-partition $\pi^*$ such that $CG(H^*,\pi^*)\cong G$.
	\end{theorem}
	\begin{proof}
		Let $G$ be a graph with vertex set  $V(G)=\{v_1,v_2,\dots,v_n\}\cup \{w_1,w_2,\dots,w_t\}$, 
		where each $v_i$ has degree at least one and each $w_i$ is an isolate. Let $G'=G-\{w_1,\dots,w_t\}$. Then $|E(G)|=|E(G')|=m$ and $|E(\overline{G'})|=\overline{m}_{G'}$, where $m+\overline{m}_{G'}=\binom{n}{2}$. To construct $H^*$, we begin with a complete graph $K_n$ on vertices $\{v_1,\dots,v_n\}$, corresponding to the non-isolates of $G$. If $n$ is even, partition these vertices into $\tfrac{n}{2}$ disjoint pairs and delete the edge within each pair; if $n$ is odd, form $\tfrac{n-1}{2}$ such pairs and delete the edge within each pair and retain all edges incident with $v_n$. These $n$ vertices are referred to as the base vertices of $H^*$. The partition $\pi^*$ is initially taken as the singleton partitions $V_i=\{v_i\}$, $1\le i\le n$. We will add to the sets in the partition as we build $H^*$. \\[2mm]
			Now we consider the following cases:\\
			Case 1: Let $G \cong K_n \cup tK_1$. If $n$ is even and $t \neq 0$, we construct $H^*$ by adding each isolate vertex $w_i \in V(G)$ to $H^*$ and connecting $w_i$ to every other vertex of $H^*$. Then, the partition $\pi^*$ is extended by adding singleton sets $W_i = \{w_i\}$ for $n \leq i \leq t$. In this construction, no set $V_i$ in $\pi^*$ is a dominating set, but every pair of sets $V_i$ in $\pi^*$ forms a coalition. Therefore, the coalition graph satisfies $CG(H^*, \pi^*) \cong K_n \cup tK_1$, where $n$ is even.\\
			 If $n$ is odd and $t \neq 0$, we first extend $H^*$ by introducing a new vertex $u_n$ associated with vertex $v_n$, such that $u_n$ is adjacent to all base vertices except $v_n$. The vertex $u_n$ is assigned to a set $V_k$ in $\pi^*$, where $k \neq n$, and the edge $u_n v_k$ is removed from $H^*$. Next, for every isolate $w_i \in V(G)$, we add $w_i$ to $H^*$ and connect it to every other vertex of $H^*$. The partition $\pi^*$ is then extended by singleton sets $W_i = \{w_i\}$ for $n \leq i \leq t$. Further, none of the sets $V_i$ is a dominating set, but every pair of sets $V_i$ in $\pi^*$ forms a coalition. Thus, $CG(H^*, \pi^*) \cong K_n \cup tK_1$, where $n$ is odd.\\[2mm]
			Case 2: If $G\cong G^{'}\cup tK_1$, where $G^{'}\ncong K_n$, then
		without loss of generality, we assume that $v_n$ is not a full vertex in $G$. For each edge $v_jv_k\in E(\overline{G'})$, introduce a new vertex $v_{jk}$ adjacent to every base vertex except $v_j$ and $v_k$, and place $v_{jk}$ into some set of $\pi^*$ other than $V_j$ or $V_k$ and then remove edges between the vertices in the same set $V_i$ in $\pi^*$.
		Here it is important to note that none of the base vertices of $H^*$ is a full vertex. Finally, add each isolate $w_i\in V(G)$ to $H^*$ and insert an edge from $w_i$ to every other vertices of $H^*$. Now, extend $\pi^*$ by a singleton sets $W_i=\{w_i\}$, $1\le i\le n$. \par 
		In the resulting graph, every edge of $H^*$ is incident with either a base vertex or a dominating vertex, and no set $V_j$ of $\pi^*$ is itself a dominating set. For any edge $v_jv_k\in E(G)$. Since $v_j$ and $v_k$ collectively dominate every vertex in $H^*$, $V_j\cup V_k$ is a coalition in $H^*$. Further, if $v_jv_k\notin E(G)$, then $V_j\cup V_k$ is not a dominating set as there is no vertex in $V_j\cup V_k$ that will dominate the vertex $v_{jk}$. Each $w_i$ is dominating vertex in $H^*$, so $w_i$ is an isolate in $G$. Hence the partition 
		$\pi^*=\{V_1,\dots,V_n,W_1,\dots,W_t\}$ 
		is a $c$-partition of $H^*$ satisfying $CG(H^*,\pi^*)\cong G$.\end{proof}
        Properties of graph $H^*$ constructed above 
	\begin{table}[h]
		~~~~~~~~~~~~\begin{tabular}{|c|c|c|}\hline
			Graph $G$&$n(H^*)$&$m(H^*)$\\ \hline
			$G\cong K_n\cup tK_1$, $n$ is even& $n+t$&$\binom{n}{2}-\dfrac{n}{2}+t(n(H^*)-1)$\\ \hline
			$G\cong K_n\cup tK_1$, $n$ is odd& $n+t+1$& $\binom{n}{2}-\dfrac{n-1}{2}+n-2+t(n(H^*)-1)$\\\hline
			$G\cong G'\cup tK_1$, where $G'\not\cong K_n$&$n+\overline{m}_{G'}+t$&$\binom{n}{2}-\lfloor{\dfrac{n}{2}}\rfloor+\overline{m}_{G'}(n-3)+t(n(H^*)-1)$\\\hline
		\end{tabular}
	\end{table}   
   \section{Properties of coalition count $c(G)$ of graph $G$}
   The following observations on coalition count $c(G)$ of graph $G$ motivated us to study the properties of $c(G)$.
\begin{enumerate}
   \item The coalition number $C(G)$ and coalition count $c(G)$ of graph $G$ are not comparable.
    For example, Consider complete graph $K_n$, $C(K_n)=n>0=c(K_n)$, but for $C_4$, $c(C_4)=6> C(C_4)=4$ and for $P_4$, $c(P_4)=C(P_4)=4$.
    \item The $c$-partition used to count the coalition number and coalition count need not be unique. For example, consider $P_6=\left\{v_1,v_2,v_3,v_4,v_5,v_6\right\}$. The partition\\ $\pi_2 = \left\{ \{v_2\}, \{v_4\}, \{v_1, v_6\}, \{v_3\}, \{v_5\} \right\}$  is a \( c \)-partition of \( P_6 \). Thus, $C(P_6)=5$ but the total number of different possible coalitions is 3. Now, consider $\pi_3 = \left\{\{v_2, v_4\}, \{v_6\}, \{v_1, v_3\}, \{v_5\} \right\}$ is a  \( c \)-partition of \( P_6 \), where $c(G)=5$.
    \item The maximum number of edges among all the graphs $CG(G,\pi)$ is the coalition count of $G$.
   \end{enumerate}
   The \textit{domatic partition} is a partition of the vertex set into dominating sets. The \textit{domatic number} $d(G)$ is equal to the maximum order $k$ of the vertex partition, called domatic partition, $P=\left\{V_1,V_2,\dots,V_k\right\}$, such that every set $V_i$ is a dominating set in $G$.\\
   In a graph $G$ of order $n$, a vertex of degree $n - 1$ is called a \textit{full vertex}. A subset $V_i$ is called a \textit{singleton set} if $|V_i| = 1$. Note that any full vertex forms a singleton dominating set.
   \begin{theorem}\label{domt}
  For any graph $G$ with no isolated vertices and $f$ full vertices, $c(G)\ge d(G)-f$.
  \end{theorem}
  \begin{proof}Assume \( P = \{V_1, V_2, \dots, V_k\} \) is a domatic partition of \( G \) with \( k = d(G) \). Without loss of generality, suppose that \( V_1, V_2, \dots, V_{k-1} \) are minimal dominating sets. If any \( V_i \) (for \( 1 \leq i \leq k - 1 \)) is not minimal, we replace it with a minimal dominating set \( V_i' \subseteq V_i \), and add the vertices in \( V_i \setminus V_i' \) to \( V_k \). Since $G$ has $f$ full vertices, there will be $f$ singleton dominating sets. Now, consider a domatic-partition $P_1=\left\{\underbrace{V_1',\dots,V_f'}_{\text{ $f$ full vertices}},\underbrace{V_{f+1}',\dots,V_{k-1}'}_{\text{minimal dominating sets}},V_k'\right\}$. Then each of  $k-f-1$ minimal dominating sets can be partitioned into two non-empty, non-dominating subsets $V_{i,1}$ and $V_{i,2}$, where $i=f+1,\dots,k-1$, whose union is dominating. Thus, these minimal dominating sets contribute at least $k-f-1$ coalitions. \\[2mm]
 Consider the vertex partition $P_2=\left\{V_1',\dots,V_f',V_{f+1,1},V_{f+2,2},\dots,V_{k-1,1}, V_{k-2,2}, V_k'\right\}$. Further, if \( V_k' \) is also a  minimal dominating set, then \( V_k' \) can be partitioned into two non-empty, non-dominating subsets whose union is dominating. Therefore, by replacing $\left\{V_k'\right\}$ in $P_2$ using these two sets, we get a $c$-partition of $G$. Thus, $c(G)\ge k-f-1+1=k-f$.\\[2mm] 
Suppose \( V_k' \) is not a minimal dominating set. Let \( V_k'' \subseteq V_k' \) be a minimal dominating set, and let \( V_k'' = V_{k,1} \cup V_{k,2} \), where \( V_{k,1} \) and \( V_{k,2} \) are non-empty, non-dominating sets that together form a coalition. Define \( W = V_k' \setminus V_k'' \). If \( W \) is a dominating set, then there exist at least \( k + 1 \) disjoint dominating sets in \( G \), contradicting \( d(G) = k \). Thus, $W$ is not a dominating set in $G$. Now, by replacing $V_k'$ in $P_2$ by $V_{k,1},V_{k,2}$ and $W$, we obtain a $c$-partition of $G$. If \( W \) can form a coalition with another non-dominating set then $c(G)\ge k-f+1$. If \( W \)  does not form a coalition with any set, then consider the union \( V_{k,2} \cup W \) as a set in $c$-partition. This set remains non-dominating but forms a coalition with \( V_{k,1} \). Thus, $c(G)\ge k-f$. Therefore, $c(G)\ge k-f=d(G)-f$.
\end{proof}
The bound of Theorem~\ref{domt} is sharp for stars $K_{1,n-1}$ , for $n\ge 3$, as $c(K_{1,n-1})=1$, $d(G)=2$ and $f=1$. Also, the bound is sharp for complete graph, where $c(K_n)=0$ and $d(K_n)=f=n$.\\
   An \textit{independent set} is a set of vertices in $G$ such that no two vertices in the set are adjacent. The \textit{independence number} $\alpha(G)$ of a graph $G$ is the cardinality of the largest independent set in $G$. \\ 
  In next theorem, we characterize the graphs whose coalition count is exactly one.
  \begin{theorem}\label{cgo}
  Let $G$ be a graph with $f$ full vertices. Then $c(G)=1$ if and only if $\alpha(G)=n-f$.
  \end{theorem}
  \begin{proof}
  	Consider a graph $G$ with $f$ full vertices $U=\left\{v_1,v_2,\dots,v_f\right\}$ and $c(G)=1$. Then there exist exactly two sets $V_i$ and $V_j$ which are not dominating sets, but $V_i\cup V_j$ is a dominating set. Since $V_i$ is not a dominating set, there exists a vertex $x\in V_j$, such that $(N[x]\backslash U) \cap V_i=\phi$.\\
			Claim: $N(x)\backslash U=\phi$.\\
			Suppose $N(x)\backslash U\neq \phi$. Then by adding the vertices in $V_j\backslash (N[x])$ to $V_i$, it follows that the set  $V_i$ is not a dominating set and $V_j=N[x]$. Now consider the sets $V_{j_1}=\left\{x\right\}$ and $V_{j_2}=N(x)$. We note that the set $V_i$ form coalition with the sets $\left\{x\right\}$ and $N(x)$. Thus, $c(G)\ge 2$, a contradiction. In similar way one can prove that if there exist a vertex in $y\in V_i$ such that $N(y)\backslash U\neq \phi$, then $c(G)\ge 2$. Thus, $N(x)\backslash U=\phi$. Therefore, $\alpha(G)=n-f$.\\[2mm]
			Conversely, suppose $G$ has $f$ full vertices and $\alpha(G)=n-f$. Then, there are $n-f$ vertices which are not adjacent to each other, and all these $n-f$ vertices is of degree $f$. Thus, in order to form a coalition partition, these $n-f$ vertices  must be divided into exactly two sets $V_i$ and $V_j$ which are not dominating sets but $V_i\cup V_j$ is a dominating set. If we partition these $n-f$ vertices into at least three sets $V_i$, $V_j$ and $V_k$. Then clearly, $V_i$ do not form coalition with $V_j$ as there is no vertices either in $V_i$ or  in $V_j$ which can dominate the vertices of $V_k$. Similarly, $V_i$ cannot  form coalition with $V_k$.  Thus, $c(G)=1$.
  \end{proof}
  The union of two graphs $G_1$ and $G_2$ is the graph with vertex set $V(G_1) \cup V(G_2)$ and edge set $E(G_1)\cup E(G_2)$. It is denoted by $G_1 \cup G_2$. The join of two graphs \( G_1 \) and \( G_2 \) is the graph obtained by taking the union of \( G_1 \) and \( G_2 \), and adding edges between every vertex of \( G_1 \) and every vertex of \( G_2 \). It is denoted by \( G_1 + G_2 \).\\
		Using Theorem~\ref{cgo}, we deduce the following result:
		\begin{corollary}
			Let \( G \) be a graph of order \( n \) with $f$ full vertices. Then $c(G)=1$ if and only if $G\cong (K_f+pK_1)\cup qK_1$, where $f+p+q=n$.
		\end{corollary}
  \begin{theorem}
   Let $G$ be a graph with exactly one full vertex and $\delta(G)=1$. If $C(G)=s(s\ge 2)$, then $c(G)=s-2$. Further, $CG(G,\pi)\cong K_1\cup K_{1,s-2}$.
  \end{theorem}
  \begin{proof}
  Consider a graph $G$ with $\delta(G)=1$ and a full vertex $u_1$. For $s=2,3$, clearly result holds. For $s\ge 4$, let $u_2,u_3,\dots,u_p$ are pendant vertices attached to $u_1$. Then $\pi=\left\{\left\{u_1\right\},W\right\}$, where $W=\left\{V_1,V_2,\dots, V_{s-1}\right\}$ is the partition of the $V(G)\backslash\left\{u_1\right\}$ such that every $V_i$ is not a dominating set but forms coalition with some other set $V_j$ in $\pi$. Here it is important to note that the pendant vertices must form a single set otherwise, it is impossible to form coalition with other sets as they are independent to each other. Without loss of generality let $\left\{u_2,u_3,\dots,u_p\right\}\subseteq V_1$. Therefore, all other sets in $W$ must form coalition with the unique set  $V_1$. If not, then there exist no vertex other than $u_1$ that will dominate the vertices $u_2,u_3,\dots,u_p$. Since $\left\{u_1\right\}$ is a dominating set, it do not form coalition with any set in $W$. Thus, $c(G)=s-2$.\\
			Since $G$ has only one full vertex and $\delta(G)=1$, the coalition graph  $CG(G,\pi)$ will have exactly two components. Further, as $CG(G,\pi)$ has $s$ vertices  out of which the vertex corresponding to the set $\left\{u_1\right\}$ do not form coalition with any other vertices, it will be an isolated vertex in $CG(G,\pi)$. Let $H$ be the another component of $CG(G,\pi)$. Then the  remaining $s-1$ vertices out of which one vertex corresponding to the set $\left\{u_2,u_3,\dots,u_p\right\}$ form coalition with all $s-2$ vertices and as $c(G)=s-2$, $H\cong K_{1,s-2}$. Thus, $CG(G,\pi)\cong K_1\cup K_{1,s-2}$. \end{proof}
    A coalition partition $\pi$ of $G$ is a singleton coalition partition if every set in $\pi$ consists of a single vertex. If a graph $G$ has a singleton
coalition partition, then $G$ is referred to as a singleton-partition graph (SP-graph).\\
   Next, we relate coalition number $C(G)$ of graph $G$ with coalition count of $G$. We omit the proof as it is straight forward.
\begin{remark}\label{obf}
 For any graph $G$ with $f$ full vertices having coalition number $C(G)$ and coalition count $c(G)$, $c(G)\ge \lceil{\dfrac{C(G)-f}{2}}\rceil$.
\end{remark}
   \begin{theorem}\label{colalpha}
   For any SP-graph $G$ of order $n$ with no full vertices, $c(G)\ge \alpha(G)$.
   \end{theorem}
   \begin{proof}
   Let $G$ is a singleton-partition graph. Then $\pi=\left\{\left\{v_1\right\},\left\{v_2\right\},\dots,\left\{v_n\right\}\right\}$ is a $c$-partition. Since $G$ has no full vertex, each set $V_i$ must form a coalition with some set $V_j$, $1\le i<j\le n$. By remark~\ref{obf}, $c(G)\ge \lceil{\dfrac{n}{2}}\rceil$.  
   If $\alpha(G)\le \lceil{\dfrac{n}{2}}\rceil$, then clearly result holds. Suppose $\alpha(G)\ge \lceil{\dfrac{n}{2}}\rceil+1$. Then there are at least $\lceil{\dfrac{n}{2}}\rceil+1$ vertices in $G$ are independent. Let $I=\left\{v_1,v_2,\dots,v_{\alpha}\right\}$ is the maximum independent set. Further, it is easy to observe that any set in $\pi$ corresponding to the vertices of $I$ will form coalition with the set corresponding to the vertex not in $I$. Thus, there will be at least $\alpha$ coalitions. Therefore, $\alpha(G)\le c(G)$.\end{proof}

\end{document}